\documentclass[11pt]{amsart}

\usepackage{amsfonts}
\usepackage{hyperref}
\usepackage{amsmath,amssymb,amsthm}
\usepackage{enumerate,graphicx,mathrsfs}
\usepackage[square,numbers]{natbib}

\newtheorem{theorem}{Theorem}[section]

\newtheorem{proposition}[theorem]{Proposition}

\theoremstyle{definition}

\newtheorem{definition}[theorem]{Definition}
\newtheorem{example}[theorem]{Example}
\newtheorem{problem}[theorem]{Problem}
\newtheorem{remark}[theorem]{Remark}

\begin{document}

\title{Characterization for entropy of shifts of finite type on Cayley trees}



\author{Jung-Chao Ban}
\address[Jung-Chao Ban]{Department of Applied Mathematics, National Dong Hwa University, Hualien 970003, Taiwan, R.O.C.}
\email{jcban@gms.ndhu.edu.tw}

\author[Chih-Hung Chang]{Chih-Hung Chang*}
\thanks{*To whom correspondence should be addressed}
\address[Chih-Hung Chang]{Department of Applied Mathematics, National University of Kaohsiung, Kaohsiung 81148, Taiwan, R.O.C.}
\email{chchang@nuk.edu.tw}

\date{May 1, 2017}

\begin{abstract}
The notion of tree-shifts constitutes an intermediate class in between one-sided shift spaces and multidimensional ones. This paper proposes an algorithm for computing of the entropy of a tree-shift of finite type. Meanwhile, the entropy of a tree-shift of finite type is $\dfrac{1}{p} \ln \lambda$ for some $p \in \mathbb{N}$, where $\lambda$ is a Perron number. This extends Lind's work on one-dimensional shifts of finite type. As an application, the entropy minimality problem is investigated, and we obtain the necessary and sufficient condition for a tree-shift of finite type being entropy minimal with some additional conditions.
\end{abstract}
\maketitle
\baselineskip=1.2 \baselineskip

\section{Introduction}

A one-dimensional shift space is a set consisting of right-infinite or bi-infinite words which avoid words in a so-called \emph{forbidden set} $\mathcal{F}$.
Such a shift space is denoted by $\mathsf{X}_{\mathcal{F}}$. It is no doubt that the most interesting $\mathsf{X}_{\mathcal{F}}$ which has been extensively investigated is the shift of finite type (SFT). That is, the shift space $\mathsf{X}_{\mathcal{F}}$ such that $\mathcal{F}$ is a finite set.

SFT plays an important role in symbolic dynamical systems and some other fields. For example, Shannon used it as a model of discrete communication channels \cite{Shannon-Bstj1949} in information theory. In dynamical systems, due to the work of Adler and Weiss \cite{Adler-1970} and Bowen \cite{Bowen-AJoM1970} that any expansive map admits a Markov partition, SFTs are used to study the hyperbolic dynamics and the classification theory of the Anosov and Axiom A diffeomorphisms. Meanwhile, SFT is also a useful tool for investigating the Lipschitz equivalence and thermodynamic properties in fractal geometry \cite{F-2003, RRW-TotAMS2012}.

A significant invariant of the SFTs is its topological entropy, which measures the growth rate of the number of the admissible patterns. Such an invariant reflects the complexity on its own right, we refer readers to \cite{ASY-1997} for more details. For a one-dimensional SFT, since its topological entropy is the logarithm of the spectral radius of a certain non-negative integral matrix \cite{LM-1995}, the topological entropy are thus being easily calculated. Such calculation method also leads to the further classification of the SFTs. The algebraic characterization of the topological entropy for SFTs is given by D.~Lind \cite{Lind-ETDS1984}, which reveals that such numbers are Perron numbers. (Recall that a Perron number is a real algebraic integer greater than $1$ and greater than the modulus of its algebraic conjugates.) More precisely, the entropies of one-dimensional SFTs are exactly the non-negative rational multiples of logarithms of Perron numbers.

The scenario for multidimensional cases is dramatically different. For example, unlike the one-dimensional case, the computations of the topological entropies for multidimensional SFTs are difficult and there is no general method. Very few models one can compute their rigorous value of entropy. Some approximation algorithms can be found in \cite{BL-DCDS2005, Friedland-2003, MP-ETDS2013, MP-SJDM2013, MP-1979}. For the entropy classification theory, the celebrated result of Hochman and Meyerovitch \cite{HM-AoM2010} indicated that the topological entropy of a multidimensional SFT is right recursively enumerable (RRE for short). Roughly speaking, it is the infimum of a monotonic recursive sequence of rational numbers. Such result has been improved by Hochman \cite{Hochm-IM2009} to the multidimensional effective dynamical systems. Later, Pavlov and Schraudner \cite{PS-TAMS2015} showed that, for every $d\geq 3$ and every $\mathbb{Z}^{d} $ full shift, there is a block gluing $\mathbb{Z}^{d}$ SFT which shares identical topological entropy.

It is worth pointing out that the reason which causes these differences between one- and multidimensional shift spaces are the spatial structure. One-dimensional lattice $\mathbb{Z}$ is a free group with one generator while multidimensional lattice $\mathbb{Z}^{d}$, $d\geq 2$, is an abelian group with $d$ generators. Thus, the investigation of symbolic dynamics on Cayley trees arises naturally. Aubrun and B\'{e}al \cite{AB-TCS2012, AB-TCS2013} introduced the notion of tree-shifts, which is a special kind of Cayley trees, and then studied the classification theory up to conjugacy, languages, and its application to automaton theory. An important point to note here is that such a shift constitutes an intermediate class in between one-sided shifts and multidimensional ones. Thus, it sheds some new light on the study of multidimensional SFTs. In \cite{BC-2015a}, the formal definition of the entropy of a tree-shift was given and the authors demonstrated that the computation of entropy of a tree-shift of finite type (TSFT) is equivalent to solving a system of nonlinear recursive equations (SNRE), and vice versa. Note that the computation of the rigorous value of entropy is not easy due to the doubly exponential growth rate of the patterns for a TSFT (see Section 2 for more details). Some partial results can be found in \cite{BC-2015a}.

The aim of this paper is to investigate classification theory of entropy. Our theorem below provides a natural and complete characterization of the entropy of TSFTs in an algebraic viewpoint.

\begin{theorem}\label{Thm: 2}
The set of entropies of tree-shifts of finite type is
\begin{equation}
\mathcal{E} = \left\{\frac{1}{p} \ln \lambda: \lambda \in \mathcal{P}, p \geq 1 \right\},  \label{35}
\end{equation}%
where $\mathcal{P}$ stands for the set of Perron numbers.
\end{theorem}

One may ask whether there is some computation method for the entropy of TSFTs. The affirmative solution is given in Theorem \ref{Thm: 4}, which says that the entropy is equal to the maximal value of the spectral radius of the adjacency matrices induced from $\mathsf{X}_{\mathcal{F}}$, we refer the reader to Section 2.3 for more details of the (reduced) system nonlinear recursive equations and the corresponding adjacency matrix.

\begin{theorem}\label{Thm: 4}
Let $X$ be a TSFT and let $F$ be the corresponding system of nonlinear recursive equations
which is defined in \eqref{27}. Then 
\begin{equation}
h(X)=\max \{\ln \lambda _{M_{E}}: E\text{ is a reduced SNRE of }%
F\}\text{.}  \label{12}
\end{equation}
\end{theorem}

Roughly speaking, the entropy of $X$ is attained on the entropy of some subsystems of itself. This makes the differences between the classical one-dimensional SFTs and TSFTs. For the convenience of the reader, we give a table for the computation method and characterization of the entropy. Let $\Omega$ be a $1$-dimensional SFT with the adjacency matrix $M$ and let $\lambda$ be its spectral radius.

\begin{center}
\begin{tabular}[t]{l|ccc}
\rule[-1ex]{0pt}{2.5ex} Entropy $\backslash$ dimension & $1$-d SFT & TSFT & $r$-d SFT \\ 
\hline
\rule[-1ex]{0pt}{3.5ex} Formula & $\ln \lambda _{M}$ & Theorem \ref{Thm: 4} & None \\
\rule[-1ex]{0pt}{2.5ex} Algebraic criterion & Perrons & Theorem \ref{Thm: 2} & None \\
\rule[-1ex]{0pt}{2.5ex} Computational criterion & Computable & Computable & RRE
\end{tabular}
\end{center}

From the table one can see that for the algebraic characterization of the entropy for multidimensional SFT is still lacking since there is no general method for the entropy computation. Our method herein sheds some new light on it due to the fact that TSFT is an intermediate class between $1$-d and $r$-d SFTs. After revealing the entropy computation algorithm, we use it to investigate which TSFT is entropy minimal.

It is known that an irreducible $\mathbb{Z}$ SFT is entropy minimal; that is, any proper subshift $Y \subset X$ has smaller entropy than that of an irreducible SFT $X$. For $r \geq 2$, every $\mathbb{Z}^r$ SFT having the mixing property called \emph{uniform filling property} is entropy minimal while there is a non-trivial \emph{block gluing} $\mathbb{Z}^r$ SFT which is not entropy minimal. Readers are referred to \cite{BPS-TAMS2010,LM-1995,QS-ETDS2003} for more details. Recently, it is demonstrated that the dimension minimality of self-affine sets holds for a generic choice in arbitrary dimension. More specifically, let $E_{\mathbf{A, v}} = \bigcup\limits_{i=1}^K A_i(E_{\mathbf{A, v}}) + v_i$ be a self-affine set corresponding to $\mathbf{A} = (A_1, \ldots, A_K) \in GL_r(\mathbb{R})^K$ and $\mathbf{v} = (v_1, \ldots, v_K) \in (\mathbb{R}^r)^K$ with $\| A_i \| < 1$ for all $i$. A folklore conjecture asserts that $\dim_H E_{\mathbf{A', v'}} < \dim_H E_{\mathbf{A, v}}$, where $\mathbf{A}'= (A_1, \ldots, A_{K-1})$ and $\mathbf{v}'= (v_1, \ldots, v_{K-1})$. There exist simple counter- examples showing that this cannot be the case for all self-affine sets; however, the conjecture holds in arbitrary dimension for a generic choice of the matrix tuple. See \cite{Fal-MPCPS1988,FM-F2007,FK-DCDS2011,KL-2017,KM-2016} and the references therein.

Suppose that $\mathsf{X}_{\mathcal{F}}$ is a Markov tree-shift (defined later) with $\mathcal{F} = \{u_1, \ldots, u_K\}$ for some $K \in \mathbb{N}$. Proposition \ref{prop:saving-symbol-entropy-minimality} reveals the necessary and sufficient condition for $\mathsf{X}_{\mathcal{F}}$ being entropy minimal. More precisely, $h(\mathsf{X}_{\mathcal{F}'}) < h(\mathsf{X}_{\mathcal{F}})$, where $\mathcal{F}' = \{u_1, \ldots, u_{K-1}\}$.

The rest of this paper is structured as follows. In Section 2, we set up the notation and terminology of the TSFTs; previous results in the computation of the entropy of TSFTs which are useful for the proof of Theorem \ref{Thm: 2} are also presented therein. Section 2.4 applies Theorem \ref{Thm: 4} to investigate some restricted entropy minimality problem and reveals the necessary and sufficient condition. The proofs of Theorem \ref{Thm: 2} and Theorem \ref{Thm: 4} are presented in Section 3.

\section{Definitions and Previous results}

This section collects some basic definitions of symbolic dynamics on Cayley trees.

\subsection{Basic definitions}

Let $\Sigma =\{0,1,\ldots ,d-1\}$ and let $\Sigma ^{\ast }=\bigcup_{n\geq 0}\Sigma ^{n}$ be the union of finite words over $\Sigma $, where $\Sigma^{n}=\{w_{1}w_{2}\cdots w_{n}:w_{i}\in \Sigma \text{ for }1\leq i\leq n\}$ is the collection of words of length $n$ for $n\in \mathbb{N}$ and $\Sigma^{0}=\{\epsilon \}$ consists of the empty word $\epsilon $. An \emph{infinite tree} $t$ over a finite alphabet $\mathcal{A}$ is a function from $\Sigma ^{\ast }$ to $\mathcal{A}$. Denote by a \emph{node of an infinite tree} a word of $\Sigma ^{\ast }$ and the empty word relates to the root of the tree. Suppose $x$ is a node of a tree. $x$ has children $xi$ with $i\in \Sigma $. A sequence of words $(x_{k})_{1\leq k\leq n}$ is called a \emph{path} if, for all $k\leq n-1$, $x_{k+1}=x_{k}i_{k}$ for some $i_{k}\in \Sigma $. Suppose $t$ is a tree and let $x$ be a node, we refer $t_{x}$ to $t(x)$ for simplicity. A subset of words $L\subset \Sigma ^{\ast }$ is called \emph{prefix-closed} if each prefix of $L$ belongs to $L$. A function $u$ defined on a finite prefix-closed subset $L$ with codomain $\mathcal{A}$ is called a \emph{pattern}, and $L$ is called the \emph{support} of the pattern. A subtree of a tree $t$ rooted at a node $x$ is the tree $t^{\prime }$ satisfying $t_{y}^{\prime }=t_{xy}$ for all $y\in \Sigma^{\ast }$ such that $xy$ is a node of $t$, where $xy=x_{1}\cdots x_{m}y_{1}\cdots y_{n}$ means the concatenation of $x=x_{1}\cdots x_{m}$ and $y = y_{1}\cdots y_{n}$.

Suppose $n\in \mathbb{N} \cup \{0\}$, $\Sigma
_{n}=\bigcup_{k=0}^{n}\Sigma ^{k}$ denotes the set of words of length at most 
$n$. We say that a pattern $u$ is \emph{a block of height $n$} (or \emph{$n$%
-block}) if the support of $u$ is $\Sigma _{n-1}$, denoted by $\mathrm{height%
}(u)=n$. Furthermore, $u$ is a pattern of a tree $t$ if there exists $x\in
\Sigma ^{\ast }$ such that $u_{y}=t_{xy}$ for every node $y$ of $u$, and say
that $u$ is a pattern of $t$ rooted at the node $x$ in this case. A tree $t$
is said to \emph{avoid} $u$ if $u$ is not a pattern of $t$. If $u$ is a
pattern of $t$, then $u$ is called an \emph{allowed pattern} of $t$.

Denote by $\mathcal{T}$ the set of all infinite trees over $\mathcal{A}$. For $%
i\in \Sigma $, the shift transformations $\sigma _{i}$ from $\mathcal{T}$ to
itself are defined as follows. For every tree $t\in \mathcal{T}$, $\sigma
_{i}(t)$ is the tree rooted at the $i$th child of $t$, that is, $(\sigma_{i}(t))_{x}=t_{ix}$ for all $x\in \Sigma ^{\ast }$. For the simplification
of the notation, we omit the parentheses and denote $\sigma _{i}(t)$ by $%
\sigma _{i}t$. The set $\mathcal{T}$ equipped with the shift transformations 
$\sigma _{i}$ is called the \emph{full tree-shift} of infinite trees over $%
\mathcal{A}$. Suppose $w=w_{1}\cdots w_{n}\in \Sigma ^{\ast }$. Define $%
\sigma _{w}=\sigma _{w_{n}}\circ \sigma _{w_{n-1}}\circ \cdots \circ \sigma
_{w_{1}}$. It follows immediately that $(\sigma _{w}t)_{x}=t_{wx}$ for all $%
x\in \Sigma ^{\ast }$.

Given a collection of patterns $\mathcal{F}$, let $\mathsf{X}_{\mathcal{F}}$ denote
the set of all trees avoiding any element of $\mathcal{F}$. A subset $%
X\subseteq \mathcal{T}$ is called a \emph{tree-shift} if $X=\mathsf{X}_{\mathcal{F}}$
for some $\mathcal{F}$. We say that $\mathcal{F}$ is \emph{a set of
forbidden patterns} of $X$. A tree-shift $X=\mathsf{X}_{\mathcal{F}}$ is called a 
\emph{tree-shift of finite type} if the forbidden set $\mathcal{F}$ is
finite. Denote by $B_{n}(X)$ the set of all blocks of height $n$ of $X$, and 
$B(X)$ the set of all blocks of $X$. Suppose $u\in B_{n}(X)$ for some $n\geq
2$. Let $\sigma _{i}u$ be the block of height $n-1$ such that $(\sigma
_{i}u)_{x}=u_{ix}$ for $x\in \Sigma _{n-2}$. The block $u$ is written as $%
u=(u_{\epsilon },\sigma _{0}u,\sigma _{1}u, \ldots, \sigma_{d-1} u)$.

Suppose $X$ and $Y$ are two one-dimensional shift spaces, the Curtis-Lyndon-Hedlund theorem (see \cite{Hed-MST1969}) indicates that a map $\phi: X \to Y$ is a sliding block code if and only if $\phi$ is continuous and $\phi \circ \sigma_x = \sigma_Y \circ \phi$. A similar discussion extends to tree-shifts; in other words, $\phi$ is a sliding block code (between tree-shifts) if and only if $\phi$ is continuous and commutes with all tree-shift maps $\sigma_i$ for $i \in \Sigma$.

If a sliding block code $\phi: X \to Y$, herein $X$ and $Y$ are tree-shifts, is onto, then $\phi$ is called a \emph{factor code} from $X$ to $Y$. A tree-shift $Y$ is a \emph{factor} of $X$ if there is a factor code from $X$ onto $Y$. If $\phi$ is one-to-one, then $\phi$ is called an \emph{embedding} of $X$ into $Y$. A sliding block code $\psi: Y \to X$ is called an \emph{inverse} of $\phi$ if $\psi(\phi(x)) = x$ for all $x \in X$ and $\phi(\psi(y)) = y$ for all $y \in Y$. In this case, we say that $\phi$ is \emph{invertible} and write $\psi = \phi^{-1}$.

\begin{definition}
A sliding block code $\phi: X \to Y$ is a \emph{conjugacy from $X$ to $Y$} if it is invertible. Two tree-shifts $X$ and $Y$ are called \emph{conjugate}, denoted by $X \cong Y$, if there is a conjugacy from $X$ to $Y$.
\end{definition}

A TSFT $X = \mathsf{X}_{\mathcal{F}}$ is called a \emph{Markov tree-shift} if the forbidden set $\mathcal{F}$ consists of $2$-blocks. In \cite{BC-2015}, Ban and Chang showed that every TSFT is conjugated to a Markov tree-shift. Therefore, it suffices to investigate Markov tree-shifts for characterizing the properties of TSFTs.


\subsection{Entropy}

Let $X$ be a tree-shift. The \emph{entropy} of $X$ is defined as follows.

\begin{definition}
\begin{enumerate}
\item The \emph{entropy} of $X$, denoted by $h(X)$, is defined as 
\begin{equation}
h(X)=\lim_{n\rightarrow \infty }\frac{\ln ^{2}|B_{n}(X)|}{n}  \label{26}
\end{equation}%
whenever the limit exists, where $|\cdot |$ stands for the cardinality of a
set and $\ln ^{2}=\ln \circ \ln $.

\item If $|B_{n}(X)|$ behaves like $\exp (\alpha \kappa ^{n})$%
, such as $|B_{n}(X)|$ $\approx c\exp (\alpha \kappa ^{n})$
for instance, where $c$ is a constant, then the value $\alpha $ is called
the \emph{hidden entropy} (or \emph{sub-entropy}) of $X$.
\end{enumerate}
\end{definition}

This paper provides an algorithm for the computation of entropy, i.e., the value $\kappa $, and gives a complete characterization of such a value. One question still unanswered is whether the same results hold for the hidden entropy $\alpha $. This question is at present far from being solved, even for the simplest cases. Namely, the case where $(d,k)=(2,2)$, the description for the hidden entropy $\alpha $ of such $X$ is still lacking. However, the computation of the exact values of $|B_{n}(X)|$ relies on the both values $\alpha $ and $\kappa$.

We introduce the notion of \emph{system of nonlinear recursive equations} which is useful for the computation of the entropy.

\begin{definition}\label{Def: 2}
Let $\mathcal{A}=\{a^{(1)},a^{(2)},\ldots ,a^{(k)}\}$ be the
symbol set and suppose $\mathcal{A}^{d}$ is an ordered set with respect to
the lexicographic order, $d\in \mathbb{N}.$
\begin{enumerate}
\item Let $F=\sum_{\mathbf{a}\in \mathcal{A}^{d}}f_{\mathbf{a}}\mathbf{a}$
be a binary combination over $\mathcal{A}^{d}$, i.e., $f_{\mathbf{a}}\in
\{0,1\}$ for $\mathbf{a}\in \mathcal{A}^{d}$. The vector $v_{F}=(f_{\mathbf{a%
}})_{\mathbf{a}\in \mathcal{A}^{d}}\in \mathbb{R}^{k^{d}}$ is called the 
\emph{indicator vector} of $F$.
\item A sequence $\{a_{n}^{(1)},\ldots ,a_{n}^{(k)}\}_{n\in \mathbb{N}}$ is
defined by a \emph{system of nonlinear recursive equations (SNRE) }of degree 
$(d,k)$ if 
\begin{equation*}
a_{n}^{(i)}=F^{(i)} \text{ for }n\geq 2,\text{ }1\leq i\leq k\text{,}
\end{equation*}%
and $a_{1}^{(i)} \in \mathbb{N}$ is given for $1\leq i\leq k$, where $F^{(1)},\ldots
,F^{(k)} $ are binary combinations over $\{a_{n-1}^{(1)},a_{n-1}^{(2)},%
\ldots ,a_{n-1}^{(k)}\}^{d}$, respectively.
\item A symbol $a^{(i)}\in \mathcal{A}$ is called \emph{essential }if there
exists an $m\in \mathbb{N}$ such that $a_{m}^{(i)}\geq 2$; otherwise, $a^{(i)}=1$ is called \emph{inessential}.
\item Suppose $F = \{F^{(i)}\}_{i=1}^k$ defines an SNRE. The \emph{indicator matrix} $I_F \in \mathbf{M}_{k \times k^d}$ of $F$ is defined as
\begin{equation}
I_F=(v_{F^{(i)}})_{i=1}^{k}.  \label{3}
\end{equation}
\end{enumerate}
\end{definition}

It is remarkable that an SNRE defined by $F$ induces a unique indicator matrix $I_{F}$, and vice versa (up to permutation). Furthermore, each $F^{(i)}$ is seen as an ordered binary combination. For example, consider the symbol set $\mathcal{A}=\{a^{(1)},a^{(2)}\}$ and the following SNRE.
\begin{equation}
\left\{ 
\begin{array}{l}
a_{n}^{(1)} =F^{(1)}=\left( a_{n-1}^{(1)}\right) ^{2}+\left(a_{n-1}^{(2)}\right) ^{2}, \\ 
a_{n}^{(2)} =F^{(2)}= a_{n-1}^{(1)} a_{n-1}^{(2)} + a_{n-1}^{(2)} a_{n-1}^{(1)}, \\ 
a_{1}^{(1)} =a_{1}^{(2)}=2.%
\end{array}%
\right.  \label{14}
\end{equation}
Then the corresponding indicator matrix is
\begin{equation*}
I_{F}= \begin{pmatrix}
1 & 0 & 0 & 1 \\ 
0 & 1 & 1 & 0
\end{pmatrix}.
\end{equation*}

For the rest of this paper, we simply use $F$ to describe the SNRE of $\{a_{n}^{(1)},\ldots ,a_{n}^{(k)}\}_{n=1}^{\infty}$. Suppose $X = \mathsf{X}_{\mathcal{F}}$ is a TSFT over $\mathcal{A} = \{a^{(1)}, a^{(2)}, \ldots, a^{(k)}\}$. Let
$$
X_{a^{(i)}} = \{t \in X: t_{\epsilon} = a^{(i)}\}
$$
be the set of those trees whose roots are assigned with the symbol $a^{(i)}$, and let $a^{(i)}_n = |B_n(X_{a^{(i)}})|$, where $1 \leq i \leq k$. Theorem \ref{Thm: 1} follows immediately.

\begin{theorem}[See \cite{BC-2015a}]\label{Thm: 1}
The values $\{a_{n}^{(1)},\ldots ,a_{n}^{(k)}\}_{n=1}^{\infty}$ satisfies the following SNRE.%
\begin{equation}
\left\{ 
\begin{array}{l}
a_{n}^{(i)}=\sum\limits_{(a^{(i)},i_{1},i_{2},\cdots ,i_{d}) \notin \mathcal{F}} \prod_{j=1}^{d} a_{n-1}^{(i_{j})}, \quad 1\leq i\leq k, n \geq 2, \\ 
a_{1}^{(i)}=\left\vert B_2(X_{a^{(i)}})\right\vert, \quad 1\leq i\leq k.%
\end{array}%
\right.  \label{27}
\end{equation}
\end{theorem}

Notably, the initial condition $a_{1}^{(i)}=\left\vert B_2(X_{a^{(i)}})\right\vert$ in \eqref{27} is the number of items of $F^{(i)}$ while, generally, the initial condition of an SNRE can be arbitrary. Define the \emph{entropy}, say $h(F)$, for an SNRE $F=\{F^{(i)}\}_{i=1}^{k}$ as
\begin{equation}
h(F)=\lim_{n\rightarrow \infty }\frac{\ln ^{2}\sum_{i=1}^{k}a_{n}^{(i)}}{n}%
\text{.}  \label{4}
\end{equation}%
Theorem \ref{Thm: 5} indicates that, for any TSFT $X$, there exists an SNRE $F$ such that $h(X)=h(F)$.

\begin{theorem}[See \cite{BC-2015a}]\label{Thm: 5}
The entropy of a tree-shift of finite type is realized as a system of nonlinear recurrence equations of degree $(d,k)$ for some $d,k\geq 2$. Conversely, every system of nonlinear recurrence equations of degree $(d,k)$ is corresponding to the entropy of some tree-shifts of finite type.
\end{theorem}

Let $A$ and $B\in \mathbf{M}_{m\times n}(\mathbb{Z})$. We say that $A\leq B$ if 
$A(i,j)\leq B(i,j)$ for $1\leq i\leq m$ and $1\leq j\leq n$. Define the
reduced SNRE as follows.

\subsection{Reduced SNRE}

This subsection introduces the notion of \emph{reduced SNRE} which enables us to build up a computational method for the entropy of a TSFT (Theorem \ref{Thm: 1}). Let us rewrite the SNRE \eqref{27} in the following form. 
\begin{equation}
\left\{ 
\begin{array}{l}
a_{n}^{(i)}=F^{(i)}=\sum_{j=1}^{k^{d}}\alpha _{j}^{(i)}F_{j}^{(i)}; \\ 
a_{1}^{(i)}=\sum_{j=1}^{k^{d}}\alpha _{j}^{(i)}\text{, }1\leq i\leq k.%
\end{array}%
\right.  \label{1}
\end{equation}


\begin{definition}[Reduced SNRE]\label{Def: 1}
Suppose $X$ is a TSFT. Let $F$ be the SNRE according to Theorem \ref{Thm: 1} and let $I_F \in \mathbf{M}_{k\times k^{d}}$ be its indicator matrix. We call $E$ a \emph{reduced SNRE} of $F$ if $E$ is the SNRE defined by some indicator matrix $I_{E}$ which satisfies the following conditions.
\begin{itemize}
\item[(i)] $I_{E}\leq I_{F}$;
\item[(ii)] $I_{E}$ has exactly one $1^{\prime }$s in each row;
\item[(iii)] the initial condition of $E$ is the same as $F$.
\end{itemize}
\end{definition}

For example, consider the SNRE $F=\{F^{(i)}\}_{i=1}^{2}$ defined in \eqref{14}; recall that the indicator matrix is a $2\times 4$ matrix 
\begin{equation*}
I_{F}=
\begin{pmatrix}
1 & 0 & 0 & 1 \\ 
0 & 1 & 1 & 0%
\end{pmatrix}.
\end{equation*}%
Then
\begin{equation*}
I_{E}=
\begin{pmatrix}
1 & 0 & 0 & 0 \\ 
0 & 1 & 0 & 0%
\end{pmatrix}
\end{equation*}%
defines a reduced SNRE $E=\{E^{(i)}\}_{i=1}^{2}$ as follows.
\begin{equation*}
\left\{ 
\begin{array}{l}
a_{n}^{(1)}=E^{(1)}=\left( a_{n-1}^{(1)}\right) ^{2}, \\ 
a_{n}^{(2)}=E^{(2)}=\left( a_{n-1}^{(1)}\right) \left( a_{n-1}^{(2)}\right), \\ 
a_{1}^{(1)}=a_{1}^{(2)}=2.%
\end{array}%
\right.
\end{equation*}%
We remark here that the initial condition $a_{1}^{(i)}$ of the reduced SNRE $E$ is no longer the number of the items of $E^{(i)}$ for $i=1,\ldots ,k$. If an SNRE $E$ which is defined by some indicator matrix $I_{E}$ satisfying only (ii), then we also call $E$ a reduced SNRE.

Let $F=\{F^{(i)}\}_{i=1}^{k}$ be a reduced SNRE. A $k\times k$ non-negative integral matrix $M$, called the \emph{weighted adjacency matrix} of $F$, is defined as
\begin{equation}
M(i,j)=\left\{ 
\begin{array}{ll}
m,  & \text{if }a^{(j)}\text{ appears in }F^{(i)}\text{ and the degree of }%
a_{n-1}^{(j)}\text{ is }m\text{;} \\ 
0,  & \text{otherwise.}%
\end{array}%
\right.  \label{30}
\end{equation}

\subsection{Entropy Minimality Problem}

The well-known entropy minimality problem investigates when the entropy of any proper subshift space is strictly smaller than the entropy of the original shift space. This subsection reveals the necessary and sufficient condition for the entropy minimality problem under some additional conditions.

\begin{proposition} \label{prop:essential-symbols-ln-d}
Suppose $X = \mathsf{X}_{\mathcal{F}}$ is a tree-shift of finite type over $\mathcal{A} = \{a^{(1)}, a^{(2)}, \ldots, a^{(k)}\}$ with an SNRE $F$ of degree $(d, k)$. If every symbol in $\mathcal{A}$ is essential, then $h(X) = \ln d$.
\end{proposition}
\begin{proof}
It suffices to show that there exists a reduced SNRE $E$ of $F$ such that $h(E) = \ln d$ since $h(X) \leq \ln d$ (cf.~\cite{BC-2015a}). Let $E$ be a reduced SNRE of $F$. Then the weighted adjacency matrix $M_E$ satisfies $\sum\limits_{j=1}^k M_E(i, j) = d$ for $1 \leq i \leq d$. Since every symbol is essential, Theorem \ref{Thm: 3} infers that the entropy of $E$ is $h(E) = \ln \rho_{M_E}$, where $\rho_{M_E}$ is the spectral radius of $M_E$. This completes the proof since $\rho_{M_E} = d$.
\end{proof}

Recall that a TSFT $X = \mathsf{X}_{\mathcal{F}}$ is called a Markov tree-shift if the height of each pattern in $\mathcal{F}$ is less than or equal to two. In \cite{BC-2015}, Ban and Chang demonstrated that every TSFT is topologically conjugated to a Markov tree-shift. For the rest of this subsection, without loss of generality, we consider those Markov tree-shifts $X = \mathsf{X}_{\mathcal{F}}$ over symbol set $\mathcal{A}$ such that every symbol is essential. Proposition \ref{prop:essential-symbols-ln-d} indicates that $h(X) = \ln d$.

This subsection investigates the entropy minimality problem described as follows. Let $Y = \mathsf{X}_{\mathcal{F}'}$ be a proper subspace of $X$ such that
\begin{enumerate}[(H1)]
\item $\mathcal{F} \subsetneq \mathcal{F}'$ and $\mathcal{F}' \setminus \mathcal{F}$ consists of only one pattern;
\item if $\mathcal{A}' \subsetneq \mathcal{A}$, then $Y$ is not a TSFT over $\mathcal{A}'$.
\end{enumerate}
In other words, the forbidden set of $Y$ is obtained by adding a pattern to the forbidden set of $X$, and every symbol which is seen in $X$ remains to be used in $Y$.

\begin{problem}
Under the above conditions, what can we say if $h(Y) < h(X)$?
\end{problem}

\begin{definition}
Suppose $X = \mathsf{X}_{\mathcal{F}}$ is a TSFT over $\mathcal{A}$. A symbol $a \in \mathcal{A}$ is called a \emph{saving symbol} for $X$ if, for each pattern $(\alpha, \alpha_1, \alpha_2, \ldots, \alpha_d) \notin \mathcal{F}$ such that $\alpha \neq a$, there exists $1 \leq i \leq d$ such that $\alpha_i = a$.
\end{definition}

\begin{proposition}\label{prop:saving-symbol-entropy-minimality}
Suppose $X = \mathsf{X}_{\mathcal{F}}$ is a TSFT over $\mathcal{A}$ with a saving symbol $a$. If $Y = \mathsf{X}_{\mathcal{F}'}$ is a proper subspace of $X$ satisfying (H1) and (H2), then $h(Y) < h(X)$ if and only if $a$ is an inessential symbol for $Y$.
\end{proposition}

\begin{remark}
Proposition \ref{prop:saving-symbol-entropy-minimality} can be rephrased as follows. $h(Y) < h(X)$ if and only if there is exact two patterns in $\mathcal{F}$ which start with $a$, $\mathcal{F}' = \mathcal{F} \bigcup \{(a, a_1, \ldots, a_d)\}$ with $a_i \neq a$ for some $i$, and $(a, a, \ldots, a) \notin \mathcal{F}'$. In other words, we can only remove the pattern (of height $2$) that starts with a saving symbol and make it an inessential saving symbol.
\end{remark}

\begin{proof}[Proof of Proposition \ref{prop:saving-symbol-entropy-minimality}]
Suppose that $a$ is an inessential saving symbol for $Y$. It follows immediately that $(a, \alpha_1, \ldots, \alpha_d) \in \mathcal{F}'$ if and only if $\alpha_i \neq a$ for some $1 \leq i \leq d$. Let $\widehat{F}$ be the corresponding SNRE of $Y$ and let $E$ be a reduced SNRE of $\widehat{F}$. Remark \ref{Rk: 1} and Theorem \ref{Thm: 3} infers that $h(E) = \ln \rho_A$, where $\rho_A$ is the spectral radius of $A$ and $A$ is the $(k-1) \times (k-1)$ matrix obtained by deleting the row and column indexed by $a$. Since $a$ is a saving symbol, $\sum\limits_{j=1}^{k-1} A(i, j) \leq d-1$ for $1 \leq i \leq k-1$. This demonstrates that $\rho_A \leq d-1$. Hence, $h(Y) \leq \ln (d-1) < h(X)$.

Conversely, $h(Y) < h(X)$ and Proposition \ref{prop:essential-symbols-ln-d} assert that there is a symbol $s \in \mathcal{A}$ such that $s$ is inessential for $Y$. We claim that there are exactly two patterns of height $2$ which start with $s$ and are accessible in $X$. Indeed, the assumptions (H1) and (H2) infer that there are at least two accessible patterns (of height $2$) in $X$ which start with $s$. Furthermore, $s$ is inessential for $Y$ derives that there are at most two accessible patterns in $X$ which start with $s$. The Claim then follows.

Suppose that $(s, \alpha_1, \ldots, \alpha_d), (s, \beta_1, \ldots, \beta_d) \notin \mathcal{F}$ and $(s, \alpha_1, \ldots, \alpha_d) \in \mathcal{F}'$. Since $s$ is an inessential symbol for $Y$, it is seen that $\beta_i = s$ for $1 \leq i \leq d$. If $s \neq a$, then $a$ being a saving symbol concludes that $\beta_i = a$ for some $1 \leq i \leq d$. The essentiality of $a$ infers that $s$ is essential for $Y$, which gets a contradiction. The proof is then complete.
\end{proof}

\section{Proofs of Main Results}

This section is dedicated to the proofs of Theorems \ref{Thm: 2} and \ref{Thm: 4}. Some useful results are presented herein. Proposition \ref{Prop: 1} is a useful tool to compute $h(F)$.

\subsection{Weighted adjacency matrix and its sprctral radius}

\begin{proposition}\label{Prop: 1}
Let $F$ be an SNRE, then 
\begin{equation}
h(F)=\lim_{n\rightarrow \infty }\frac{\ln \sum_{i=1}^{k}\ln a_{n}^{(i)}}{n}.  \label{31}
\end{equation}
\end{proposition}

\begin{proof}
Since for every $n,k\in \mathbb{N}$ ,
\begin{equation*}
a_{n}^{(1)}a_{n}^{(2)}\cdots a_{n}^{(k)}\leq \left( \frac{%
\sum_{i=1}^{k}a_{n}^{(i)}}{k}\right) ^{k},
\end{equation*}%
we derive that 
\begin{equation*}
\sum_{i=1}^{k}\ln a_{n}^{(i)}\leq k\left( \ln \sum_{i=1}^{k}a_{n}^{(i)}-\ln
k\right)
\end{equation*}%
and 
\begin{equation*}
\lim_{n\rightarrow \infty }\frac{\ln \sum_{i=1}^{k}\ln a_{n}^{(i)}}{n}\leq
\lim_{n\rightarrow \infty }\frac{\ln ^{2}\sum_{i=1}^{k}a_{n}^{(i)}}{n}=h(F)%
\text{.}
\end{equation*}%
Conversely, let $a_{n}=\max_{1\leq i\leq k}a_{n}^{(i)}$. The inequality 
\begin{equation*}
a_{n}\leq \sum_{i=1}^{k}a_{n}^{(i)}\leq ka_{n},
\end{equation*}%
yields that 
\begin{equation}
\lim_{n\rightarrow \infty }\frac{\ln ^{2}\sum_{i=1}^{k}a_{n}^{(i)}}{n}%
=\lim_{n\rightarrow \infty }\frac{\ln ^{2}a_{n}}{n}\text{.}  \label{28}
\end{equation}%
On the other hand, we have 
\begin{equation}
\sum_{i=1}^{k}\ln a_{n}^{(i)}=\ln \prod_{i=1}^{k}a_{n}^{(i)}\geq \ln
\max_{1\leq i\leq k}a_{n}^{(i)}=\ln a_{n}\text{.}  \label{29}
\end{equation}%
Combining \eqref{28} with \eqref{29} concludes that 
\begin{equation*}
\lim_{n\rightarrow \infty }\frac{\ln \sum_{i=1}^{k}\ln a_{n}^{(i)}}{n}\geq
\lim_{n\rightarrow \infty }\frac{\ln ^{2}a_{n}}{n}=\lim_{n\rightarrow \infty
}\frac{\ln ^{2}\sum_{i=1}^{k}a_{n}^{(i)}}{n}=h(F)\text{.}
\end{equation*}%
The proof is thus complete.
\end{proof}

\begin{remark}\label{Rk: 1}
Suppose we partition the symbol set $\mathcal{A}$ as
\begin{equation}
\mathcal{A}=\mathcal{A}_{E}\cup \mathcal{A}_{I},  \label{38}
\end{equation}%
where $\mathcal{A}_{E}$ is the collection of the essential symbols in $\mathcal{A}$ and $\mathcal{A}_{I}$ collects the inessential symbols (Definition \ref{Def: 2}), it follows from Proposition \ref{Prop: 1} that
\begin{equation*}
h(F)=\lim_{n\rightarrow \infty }\frac{\ln \sum_{i=1}^{k}\ln a_{n}^{(i)}}{n}%
=\lim_{n\rightarrow \infty }\frac{\ln \sum_{a^{(i)}\in \mathcal{A}%
_{E}}a_{n}^{(i)}}{n}\text{.}
\end{equation*}%
That is, the entropy $h(F)$ is the growth rate of the sum of all essential symbols. In this case, we say that $h(F)$ is \emph{supported} on $\mathcal{A}_{E}.$
\end{remark}

Let $\Omega $ be a one-dimensional subshift of finite type and let $A=A_{\Omega} $ be the corresponding adjacency matrix, the classical result in symbolic dynamics shows that the topological entropy of $\Omega$ is $h(\Omega)=\ln \lambda_{A}$, where $\lambda_A$ is the maximal eigenvalue of $A$ (cf.~\cite{LM-1995}). Theorem \ref{Thm: 3} is an analogous result for reduced SNREs.

\begin{theorem}\label{Thm: 3}
Let $F=\{F^{(i)}\}_{i=1}^{k}$ be a reduced SNRE and let $M$ be the
corresponding weighted adjacency matrix which is defined in \eqref{30}. If there exists $N \in \mathbb{N}$ such that $a_{n}^{(i)}>1$ for all $i=1,\ldots ,k$ and $n \geq N$, then 
\begin{equation*}
h(F)=\ln \lambda _{M}\text{,}
\end{equation*}%
where $\lambda _{M}$ is the spectral radius of $M$.
\end{theorem}

\begin{proof}
Let $F$ be a reduced SNRE. That is, $F$ is defined by an indicator matrix $%
I_{F}$ which satisfies the condition (ii) of Definition \ref{Def: 1}. We write the
SNRE $F$ in the following form. 
\begin{equation*}
F=\{F^{(i)}=\left( a_{n-1}^{(1)}\right) ^{m_{1}^{(i)}}\left(
a_{n-1}^{(2)}\right) ^{m_{2}^{(i)}}\cdots \left( a_{n-1}^{(k)}\right)
^{m_{k}^{(i)}}\}\text{,}
\end{equation*}%
where $(m_{1}^{(i)},\ldots ,m_{k}^{(i)})$ is a non-negative integral $k$%
-tuple for all $i$. Define 
\begin{equation*}
b_{n}:=(\ln a_{n}^{(1)},\ldots ,\ln a_{n}^{(k)})^{T}\text{.}
\end{equation*}%
It is seen that $b_{n}=Mb_{n-1}$. Combining the facts of $b_{n}=M^{n-1}b_{1}$, \eqref{31}, and $a_{n}^{(i)} > 1$ for $n$ large enough yields that
\begin{equation*}
h(F)=\lim_{n\rightarrow \infty }\frac{\ln \sum_{i=1}^{k}\ln a_{n}^{(i)}}{n}%
=\lim_{n\rightarrow \infty }\frac{\ln \sum_{i,j=1}^{k}M^{n-1}(i,j)}{n}=\ln
\lambda _{M}\text{.}
\end{equation*}%
This completes the proof.
\end{proof}

\subsection{Proof of Theorem \protect\ref{Thm: 4}}

Theorem \ref{Thm: 3} reveals that the computation of the entropy of a reduced SNRE is analogous to the classical result of SFTs. Theorem \ref{Thm: 4} provides the method for the computation of $h(F)$ for general $F$; that is, the entropy of $h(X)$ (Theorem \ref{Thm: 2}). The proof of Theorem \ref{Thm: 4} is presented herein.

\begin{proof}[Proof of Theorem \protect\ref{Thm: 4}]
Set 
\begin{equation*}
h=\max \{\ln \lambda _{M_{E}}:E\text{ is reduced from }F\}\text{.}
\end{equation*}%
Let $b_{n}^{(i)}=\ln a_{n}^{(i)}$ for all $1\leq i\leq k$ and $F^{(i)}$ be
arranged as the following form. 
\begin{equation}
F^{(i)}=\sum_{j=1}^{r_{i}}\beta _{j}^{(i)}F_{j}^{(i)}, \quad \beta _{j}^{(i)} \neq 0, i=1,\ldots, k.  \label{23}
\end{equation}%
Since the computation of $h(F)$ is supported on those essential symbols (see Remark \ref{Rk: 1}), without loss of generality, we assume that $b_{1}^{(i)}\geq 2$ for $i=1,\ldots ,k$. The existence of the limit of $h(X)=h(F)$ infers that there is a subsequence $\{a_{n_{\ell}}^{(j_{i})}\}_{i=1}^{k}$ satisfying
\begin{equation}
a_{n_{\ell}}^{(j_{1})}\geq a_{n_{\ell}}^{(j_{2})}\geq \cdots \geq
a_{n_{\ell}}^{(j_{k})} \quad \text{for} \quad \ell \in \mathbb{N}  \label{17}
\end{equation}%
and
\begin{equation*}
h(F)=\lim_{n_{\ell}\rightarrow \infty }\frac{\ln \sum_{i=1}^{k}\ln
a_{n_{\ell}}^{(j_{i})}}{n_{\ell}}\text{.}
\end{equation*}%
For simplicity, we may assume that $j_{i}=i$ for $i=1,\ldots ,k$ and $n_{\ell} = \ell$ for $\ell \in \mathbb{N}$. Thus, \eqref{23} can be rewritten as follows. 
\begin{equation*}
F^{(i)}=F_{1}^{(i)}\left( \beta _{1}^{(i)}+\sum_{l=2}^{r_{i}}\beta _{l}^{(i)}%
\frac{F_{l}^{(i)}}{F_{1}^{(i)}}\right) =F_{1}^{(i)}c_{n-1}^{(i)}\text{,}
\end{equation*}%
where 
\begin{equation*}
c_{n-1}^{(i)}=\beta _{1}^{(i)}+\sum_{l=2}^{r_{i}}\beta _{l}^{(i)}\frac{%
F_{l}^{(i)}}{F_{1}^{(i)}}\text{.}
\end{equation*}%
It follows from \eqref{17} that
\begin{equation}
2\leq c_{n-1}^{(i)}\leq \sum_{l=1}^{r_{i}}\beta _{l}^{(i)}\leq C\text{ for
all }n\in \mathbb{N}\text{,}  \label{20}
\end{equation}%
where 
\begin{equation*}
C=r\beta ,r:=\max_{1\leq i\leq k}r_{i}\text{ and }\beta :=\max_{1\leq i\leq
k,1\leq l\leq r_{i}}\beta _{l}^{(i)}\text{.}
\end{equation*}%
That is, 
\begin{equation}
\left\{ 
\begin{array}{c}
b_{n}^{(1)}=\ln a_{n}^{(1)}=\ln F_{1}^{(1)}+\ln c_{n-1}^{(1)}, \\ 
b_{n}^{(2)}=\ln a_{n}^{(2)}=\ln F_{1}^{(2)}+\ln c_{n-1}^{(2)}, \\ 
\vdots  \\ 
b_{n}^{(k)}=\ln a_{n}^{(k)}=\ln F_{1}^{(k)}+\ln c_{n-1}^{(k)}.%
\end{array}%
\right.   \label{18}
\end{equation}

Let $b_{n}=\left( b_{n}^{(1)},\ldots ,b_{n}^{(k)}\right) ^{T}$. Notably, for $i = 1, \ldots, k$, $F_{1}^{(i)}$ is of the form
\begin{equation*}
F_{1}^{(i)}=\left( a_{n-1}^{(1)}\right) ^{m_{1}^{(i)}}\cdots \left(a_{n-1}^{(k)}\right) ^{m_{k}^{(i)}};
\end{equation*}%
we have 
\begin{equation*}
\ln F_{1}^{(i)}=\sum_{l=1}^{k}m_{l}^{(i)}\ln
a_{n-1}^{(l)}=\sum_{l=1}^{k}m_{l}^{(i)}b_{n-1}^{(l)}\text{.}
\end{equation*}%
Thus \eqref{18} can be represented as
\begin{equation}
b_{n}=Mb_{n-1}+\ln c_{n-1}\text{,}  \label{19}
\end{equation}%
where 
\begin{equation*}
\ln c_{n}=(\ln c_{n}^{(1)},\ldots ,\ln c_{n}^{(k)})^{T},
\end{equation*}%
and $M$ is the weighted adjacency matrix of $E=\{F_{1}^{(i)}\}_{i=1}^{k}$.
Iterate \eqref{19} we obtain%
\begin{equation}
b_{n}=M^{n-1}b_{1}+M^{n-2}c_{1}+\cdots +c_{n-1}\text{.}
\end{equation}%
Let $\lambda =\lambda _{M}$ it follows from Proposition 4.2.1 of \cite%
{LM-1995} and \eqref{20} that
\begin{eqnarray*}
\sum_{i=1}^{k}b_{n}^{(i)} &=&\left\Vert b_{n}\right\Vert \leq
d_{0}\left\Vert \sum_{i=0}^{n-1}M^{i}\right\Vert \leq
d_{0}\sum_{i=0}^{n-1}\left\Vert M^{i}\right\Vert \\
&\leq &d_{1}\left( \sum_{i=0}^{n-1}\lambda ^{i}\right) \leq d_{2}\lambda
^{n},
\end{eqnarray*}%
where $d_{0}$, $d_{1}$, and $d_{2}$ only depend on the dimension of $M$ and $%
k $. Thus we have 
\begin{equation}
\sum_{i=1}^{k}b_{n}^{(i)}\leq d_{2}\lambda ^{n}\text{.}  \label{21}
\end{equation}%
Combining Proposition \ref{Prop: 1} with \eqref{21} infers that
\begin{equation*}
h(F)=\lim_{n\rightarrow \infty }\frac{\ln \sum_{i=1}^{k}b_{n}^{(i)}}{n}\leq
\lim_{n\rightarrow \infty }\frac{\ln d_{2}\lambda ^{n}}{n}=\ln \lambda \text{%
.}
\end{equation*}

Similarly, combining \eqref{20}, \eqref{19} with the fact that $%
b_{1}^{(i)}\geq 2$ for all $i$ we have 
\begin{equation*}
\sum_{i=1}^{k}b_{n}^{(i)}\geq d_{3}\lambda ^{n}\text{,}
\end{equation*}%
for some $d_{3}>0$, which implies 
\begin{equation*}
h(F)\geq \ln \lambda \text{.}
\end{equation*}

Thus $h(F)=\ln \lambda $. That is $h(F)$ is the logarithm of the spectral radius of some integral matirx $M$ which is a weighted adjacency matrix of some SNRE $E$ reduced from $F$. Thus we conclude that 
\begin{equation}
h(F)\leq h\text{.}  \label{10}
\end{equation}%
For the converse, suppose $E$ is a reduced SNRE of $F$ with $h=\ln \lambda
_{M_{E}}$. From the (i) of Definition \ref{Def: 1} we have $E^{(i)}\leq
F^{(i)}$ for all $i$. It implies that 
\begin{equation}
h(F)\geq h(E)=\ln \lambda _{M_{E}}=h\text{.}  \label{11}
\end{equation}%
Combining \eqref{10} with \eqref{11} yields \eqref{12}. The proof is thus
completed.
\end{proof}

\subsection{Proof of Theorem \protect\ref{Thm: 2}}

The proof of Theorem \ref{Thm: 2} is presented.

\begin{proof}[Proof of Theorem \protect\ref{Thm: 2}]
Let $X$ be a TSFT and let $F$ be its SNRE. Since $h(X^{%
\mathcal{B}})=h(F)$ from Theorem \ref{Thm: 1}, thus it follows from Theorem %
\ref{Thm: 4} that 
\begin{equation*}
h(X)=h(F)=h(E)=\ln \lambda _{M_{E}}=:\ln \lambda _{E}
\end{equation*}%
for some reduced SNRE $E$ of $F$. If the weighted adjacency matrix $M_E$ is primitive, the Perron-Frobenius theorem concludes that $\lambda _{E}\in \mathcal{P}$. That is, $\ln \lambda _{E}\in \mathcal{E}$ (recall \eqref{35}). If $M_E$ is irreducible or reducible, Corollary of Theorem 3 of \cite{Lind-ETDS1984} shows that that $\lambda_{E}=\lambda ^{\frac{1}{p}}$ for some $\lambda \in \mathcal{P}$ and $1<p\in \mathbb{N}$. Thus, $\ln \lambda _{E}\in \mathcal{E}$. Conversely, we set $h=\ln \lambda ^{\frac{1}{p}}\in \mathcal{E}$. If $p=1$, then Theorem 1 of \cite{Lind-ETDS1984} is applied to show that there is a primitive non-negative integral matrix with $\lambda $ is its spectral radius. More precisely, there is an $m\times m$ primitive non-negative integral matrix $M$ such that $\lambda $ is its spectral radius. Construct a new matrix $V=V_{M}$ as follows. Denote by 
\begin{equation}
d:=d(M)=\max_{1\leq i\leq k}\sum_{j=1}^{k}M(i,j)\text{.}  \label{37}
\end{equation}%
Define a $\left( k+1\right) \times \left( k+1\right) $ non-negative integral
matrix $V$ as follows.%
\begin{equation*}
V(i,j)=\left\{ 
\begin{array}{ll}
M(i,j), & \text{if }1\leq i\leq k,1\leq j\leq k\text{;} \\ 
d-\sum_{j=1}^{r}M(i,j),  & \text{if }1\leq i\leq k,j=k+1\text{;} \\ 
0, & \text{if }i=k+1,1\leq j\leq k\text{;} \\ 
d, & \text{if }i=k+1,j=k+1\text{.}%
\end{array}%
\right. 
\end{equation*}

It can be easily checked that 
\begin{equation*}
\sum_{j=1}^{k+1}V(i,j)=d\text{ for }1\leq i\leq k+1\text{.}
\end{equation*}%
Introduce the symbol set $\mathcal{A}$ and an SNRE $F=\{F^{(i)}%
\}_{i=1}^{k+1} $ according to $V$ as follows. Let 
\begin{equation*}
\mathcal{A}=\{a^{(i)}\}_{i=1}^{k}\cup \{a^{(k+1)}\}\text{.}
\end{equation*}%
For $i=1,\ldots ,k$, we define 
\begin{eqnarray*}
F^{(i)} &=&\left( a_{n-1}^{(1)}\right) ^{V(i,1)}\cdots \left(
a_{n-1}^{(k)}\right) ^{V(i,k)}\left( a_{n-1}^{(k+1)}\right)
^{d-V(i,k+1)}+\left( a_{n-1}^{(k+1)}\right) ^{d}\text{, } \\
F^{(k+1)} &=&\left( a_{n-1}^{(k+1)}\right) ^{d}\text{.}
\end{eqnarray*}

Since $\sum_{j=1}^{k}V(i,j)=d$ for all $i$, each $F^{(i)}$ is a polynomial
of degree $d$. It implies that $F=\{F^{(i)}\}_{i=1}^{k+1}$ is a $(d,k+1)$%
-SNRE with the initial conditions of $F$ is 
\begin{equation}
a_{1}^{(i)}=2\text{ for }1\leq i\leq k\text{ and }a_{1}^{(k+1)}=1\text{.}
\label{39}
\end{equation}%
Combining \eqref{39} with the fact that $a_{n}^{(k+1)}=1$ for all $n$ (since 
$a_{n}^{(k+1)}$ only connect to itself), we conclude that $a_{n}^{(k+1)}$
must be the least element with respect to the lexicographic order defined
in the proof of Theorem \ref{Thm: 4} (since $a_{1}^{(i)}\geq 2\geq
a_{1}^{(1)}$ for $i=1,\ldots ,k$ and $a_{n}^{(k+1)}=1$). Therefore, we
obtain that the entropy $h(F)$ is attained at the logarithm of the spectral
radius of the weighted adjacency matrix $M$ corresponding to a reduced SNRE $E$ of $F
$, where $E=\{E^{(i)}\}_{i=1}^{k+1}$ is as follows.%
\begin{eqnarray*}
E^{(i)} &=&\left( a_{n-1}^{(1)}\right) ^{V(i,1)}\cdots \left(
a_{n-1}^{(k)}\right) ^{V(i,k)}\left( a_{n-1}^{(k+1)}\right) ^{d-V(i,k+1)},%
\text{ }1\leq i\leq k, \\
E^{(k+1)} &=&\left( a_{n-1}^{(k+1)}\right) ^{d}\text{.}
\end{eqnarray*}%
Meanwhile, $M$ is of the form%
\begin{equation*}
M_{E}=\left( 
\begin{array}{cc}
M & U \\ 
0 & d%
\end{array}%
\right) \text{,}
\end{equation*}%
where $U$ is a $k\times 1$ matrix with entries are $U(i,k+1)=d-V(i,k+1)$ for 
$1\leq i\leq k$.

Let
\begin{eqnarray*}
\mathcal{B} &=&\left\{(a^{(i)}, \overbrace{a^{(1)}, \cdots, a^{(1)}}^{V(i,1)\text{-times}}, \cdots, \overbrace{a^{(k+1)}, \cdots, a^{(k+1)}}^{d-V(i,k+1)\text{-times}})\right\}_{i=1}^{k} \\
&&\bigcup \left\{(a^{(i)}, \overbrace{a^{(k+1)}, \cdots, a^{(k+1)}}^{d\text{-times}})\right\}_{i=1}^{k}\bigcup \left\{a^{(k+1)}, \overbrace{a^{(k+1)}, \cdots, a^{(k+1)}}^{d\text{-times}}\right\}
\end{eqnarray*}
and let $\mathcal{F} = \mathcal{A}^d \setminus \mathcal{B}$. We claim that the TSFT $X = \mathsf{X}_{\mathcal{F}}$ carries entropy $h(X)=\ln \lambda $. Indeed, it follows from Proposition \ref{Prop: 1}, Theorems \ref{Thm: 4} and \ref{Thm: 1}, and the fact of $a_{n}^{(k+1)}=1$ for all $n$ that
\begin{equation*}
h(X)=h(F)=\lim_{n\rightarrow \infty }\frac{\ln
\sum_{i=1}^{k}\ln a_{n}^{(i)}}{n}=h(E)=\ln \lambda =h.
\end{equation*}%
This shows that $h$ is a entropy of some tree SFT $X$. Hence, the claim holds. Finally, if $p>1$, i.e., $h=\ln \lambda ^{\frac{1}{p}}\in
\mathcal{E}$, we first note that the above argument of the constructing
the TSFT $X$ is also applied to the case where $M$ is a
non-negative irreducible integral matrix. Theorem 3 of \cite{Lind-ETDS1984}
says that a positive number is the spectral radius of an irreducible
non-negative integral matrix if and only if some positive integral power of
it is a Perron number. Since $\left( \lambda ^{\frac{1}{p}}\right)
^{p}=\lambda \in \mathcal{P}$, $\lambda ^{\frac{1}{p}}$ must be a a
spectral radius of some non-negative irreducible integral matrix, say $M$.
Using the same argument as above one could construct a TSFT $X$ such that $h(X)=h(F)=\ln \lambda $. The proof is thus complete.
\end{proof}

\subsection{Example}

\begin{example}
Let $h=\ln \lambda $, where $\lambda $ is the spectral radius of
\begin{equation*}
M=\left( 
\begin{array}{ccc}
1 & 1 & 0 \\ 
0 & 0 & 1 \\ 
2 & 1 & 0%
\end{array}%
\right) \in \mathbf{M}_{3\times 3}\text{.}
\end{equation*}%
Since $d=3$ (defined in \eqref{37}), we construct a $4\times 4$ matrix $V $ as
\begin{equation*}
V=\left( 
\begin{array}{cccc}
1 & 1 & 0 & 1 \\ 
0 & 0 & 1 & 2 \\ 
2 & 1 & 0 & 0 \\ 
0 & 0 & 0 & 3%
\end{array}%
\right) =\left( 
\begin{array}{cc}
M & U \\ 
0 & 3%
\end{array}%
\right).
\end{equation*}

The $(3,4)$-SNRE can be constructed as follows. 
\begin{equation*}
\left\{ 
\begin{array}{l}
a_{n}^{(1)}=a_{n-1}^{(1)}a_{n-1}^{(2)}a_{n-1}^{(4)}+\left(
a_{n-1}^{(4)}\right) ^{3},a_{2}^{(1)}=2 \\ 
a_{n}^{(2)}=a_{n-1}^{(3)}\left( a_{n-1}^{(4)}\right) ^{2}+\left(
a_{n-1}^{(4)}\right) ^{3},a_{2}^{(2)}=2 \\ 
a_{n}^{(3)}=\left( a_{n-1}^{(1)}\right) ^{2}a_{n-1}^{(2)}+\left(
a_{n-1}^{(4)}\right) ^{3},a_{2}^{(3)}=2 \\ 
a_{n}^{(4)}=\left( a_{n-1}^{(4)}\right) ^{3},a_{2}^{(4)}=1.%
\end{array}%
\right.
\end{equation*}%
Meanwhile, the corresponding set $\mathcal{B}$ can also be defined as
\begin{align*}
\mathcal{B} =\bigg\{ &( a^{(1)}, a^{(1)}, a^{(2)}, a^{(4)}), (a^{(1)}, a^{(4)}, a^{(4)}, a^{(4)}), (a^{(2)}, a^{(3)}, a^{(4)}, a^{(4)}), \\
& ( a^{(2)}, a^{(4)}, a^{(4)}, a^{(4)}), (a^{(3)}, a^{(1)}, a^{(1)}, a^{(2)}), (a^{(3)}, a^{(4)}, a^{(4)}, a^{(4)}), \\
& ( a^{(4)}, a^{(4)}, a^{(4)}, a^{(4)}) \bigg\}.
\end{align*}%
The same argument as seen in the proof of Theorem \ref{Thm: 2} shows that the TSFT $X = \mathsf{X}_{\mathcal{F}}$ with $\mathcal{F} = \mathcal{A}^4 \setminus \mathcal{B}$ is capable of the entropy $\ln \lambda$.
\end{example}

\section*{Acknowledgment}

The author would like to thank Ms.~Nai-Zhu Huang for the discussion during the preparation of this work.

\bibliographystyle{amsplain}
\bibliography{../../grece}

\end{document}